\documentclass{amsart}
\usepackage{amsmath,enumerate,amsfonts}

\newtheorem{lemma}{Lemma}[section]

\newtheorem{proposition}[lemma]{Proposition}
\newtheorem{theorem}[lemma]{Theorem}

\newtheorem{definition}[lemma]{Definition}

\newtheorem{conjecture}[lemma]{Conjecture}

\title[An inequality for matrix pressure]{An inequality for the matrix pressure function and applications}
\author{Ian D. Morris}

\begin{document}

\maketitle

\begin{abstract}
We prove an \emph{a priori} lower bound for the pressure, or $p$-norm joint spectral radius, of a measure on the set of $d \times d$ real matrices which parallels a result of J. Bochi for the joint spectral radius. We apply this lower bound to give new proofs of the continuity of the affinity dimension of a self-affine set and of the continuity of the singular-value pressure for invertible matrices, both of which had been previously established by D.-J. Feng and P. Shmerkin using multiplicative ergodic theory and the subadditive variational principle.
Unlike the previous proof, our lower bound yields algorithms to rigorously compute the pressure, singular value pressure and affinity dimension of a finite set of matrices to within an \emph{a priori} prescribed accuracy in finitely many computational steps. We additionally deduce a related inequality for the singular value pressure for measures on the set of $2 \times 2$ real matrices, give a precise characterisation of the discontinuities of the singular value pressure function for two-dimensional matrices, and prove a general theorem relating the zero-temperature limit of the matrix pressure to the joint spectral radius.
\end{abstract}
\section{Introduction}
If $A_1,\ldots,A_N$ are $d \times d$ real matrices and $s>0$ a real number, we may define the \emph{(norm) pressure} of $A_1,\ldots,A_N$ to be the quantity
\[\mathbf{M}((A_1,\ldots,A_N),s):=\lim_{n \to \infty} \frac{1}{n}\log \left(\sum_{i_1,\ldots,i_n=1}^N \left\|A_{i_1}\cdots A_{i_n}\right\|^s\right) \in [-\infty,+\infty),\]
the existence of the limit being guaranteed by subadditivity. This quantity has also been studied in the form of the $p$-norm joint spectral radius, or $p$-radius, defined by
\begin{equation}\label{eq:pradius}\varrho_p(A_1,\ldots,A_N):=\lim_{n\to \infty} \left(\frac{1}{N^n}\sum_{i_1,\ldots,i_n=1}^N \left\|A_{i_1}\cdots A_{i_n}\right\|^p\right)^{\frac{1}{np}}=N^{-\frac{1}{p}}e^{\mathbf{M}((A_1,\ldots,A_n),p)/p}\end{equation}
where it is usually assumed that $p \geq 1$. (Here, and in general throughout the paper, we adopt the conventions $\log 0:=-\infty$, $e^{-\infty}:=0$.) The norm pressure and $p$-radius have been extensively investigated for their connections with wavelet analysis \cite{Ji95,LaWa95,Wa96}, the stability of switched linear systems \cite{OgMa13}, and thermodynamic formalism and multifractal analysis \cite{Fe03,Fe09,FeLa02}; in recent years significant attention has been given to the efficient computation of the $p$-radius \cite{JuPr10,JuPr11,OgMa14,OgPrJu15,Pr97}.

In this article we shall also be concerned with a related quantity, the \emph{singular value pressure} of a finite set of matrices. Let $M_d(\mathbb{R})$ denote the vector space of all $d \times d$ real matrices, and let $\sigma_1(A),\ldots,\sigma_d(A)$ denote the singular values of a matrix $A \in M_d(\mathbb{R})$, which are defined to be the non-negative square roots of the eigenvalues of the positive semidefinite matrix $A^*A$, listed in decreasing order with repetition in the case of multiplicity. For each $s>0$ and $A \in M_d(\mathbb{R})$ we define
\[\varphi^s(A):=\left\{\begin{array}{cl}\sigma_1(A)\cdots \sigma_{k}(A)\sigma_{k+1}(A)^{s-k},&k \leq s \leq k+1 \leq d\\
|\det A|^\frac{s}{d},&s \geq d.\end{array}\right.\]
The function $\varphi$ may easily be seen to be upper semi-continuous in $(A,s)$ with discontinuities occuring precisely when $s$ is an integer from $1$ to $d-1$ such that $\sigma_{s+1}(A)=0<\sigma_s(A)$. We have $\varphi^s(AB)\leq \varphi^s(A)\varphi^s(B)$ for all $A,B \in M_d(\mathbb{R})$ and $s>0$, see e.g. \cite[Lemma 2.1]{Fa88}. For $A_1,\ldots,A_N \in M_d(\mathbb{R})$ and $s>0$ we define the \emph{singular value pressure} of $A_1,\ldots,A_N$  by
\[\mathbf{P}((A_1,\ldots,A_N),s):=\lim_{n \to \infty} \frac{1}{n}\log \left(\sum_{i_1,\ldots,i_n=1}^N \varphi^s\left(A_{i_1}\cdots A_{i_n}\right)\right) \in [-\infty,+\infty).\]
The singular value pressure plays a pivotal role in the dimension theory of self-affine fractals and has been extensively applied in that context (see e.g. \cite{Fa88,JoPoSi07,So98}). Let us recall the definition of a self-affine set. If $T_1,\ldots,T_N \colon \mathbb{R}^d \to \mathbb{R}^d$ are contractions with respect to the Euclidean metric -- that is, if there exists $\lambda \in [0,1)$ such that $\|T_ix-T_iy\| \leq \lambda \|x-y\|$ for all $x,y \in \mathbb{R}^d$ and $i=1,\ldots,N$ -- then by a well-known theorem of J. E. Hutchinson \cite{Hu81} there exists a unique nonempty compact set $X \subset \mathbb{R}^d$, called the \emph{attractor} of $T_1,\ldots,T_N$, which satisfies the equation
\[X=\bigcup_{i=1}^N T_iX.\]
If $T_1,\ldots,T_N$ are affine contractions then the resulting set $X$ is termed a \emph{self-affine set}. The following foundational result of K. Falconer \cite{Fa88} determines the Hausdorff dimension for ``typical'' self-affine sets with small enough contraction ratio:
\begin{theorem}[(Falconer)]
Let $A_1,\ldots,A_N \in M_d(\mathbb{R})$ with $\max\|A_i\|<\frac{1}{3}$, and define
\[\mathfrak{s}(A_1,\ldots,A_N):=\inf\left\{s>0 \colon \mathbf{P}((A_1,\ldots,A_N),s)<0\right\}.\]
Then for Lebesgue-almost-every $(v_1,\ldots,v_N) \in \mathbb{R}^{Nd}$ the attractor associated to the maps $T_1,\ldots,T_N \colon \mathbb{R}^d \to \mathbb{R}^d$ defined by $T_ix:=A_ix+v_i$ has Hausdorff dimension equal to $\min\{\mathfrak{s}(A_1,\ldots,A_N),d\}$.
\end{theorem}
The condition $\max\|A_i\|<\frac{1}{3}$ was subsequently weakened to $\max\|A_i\|<\frac{1}{2}$ by B. Solomyak \cite{So98}, and to  $\max\|A_i\|<1$ by Jordan, Pollicott and Simon for a modified notion of self-affine set which incorporates random perturbations \cite{JoPoSi07}. In the absence of random perturbations the constant $\frac{1}{2}$ is optimal \cite{Ed92}. The function $\mathfrak{s}$, sometimes called the \emph{affinity dimension} or \emph{singularity dimension}, arises as a formula for fractal dimension in numerous other works, of which we note for example \cite{Fa92,HuLa95,KaSh09}.

Despite the interest of these results the computation and regularity of $\mathbf{P}$ and $\mathfrak{s}$ have been investigated in depth only relatively recently \cite{FaSl09,FeSh14,Fr15,MaSi07,PoVy15}; in particular, the continuity of the function $\mathfrak{s}$ with respect to the matrices $A_1,\ldots,A_N$ was established only in 2014 \cite{FeSh14}.  In this article we prove a general inequality which gives \emph{a priori} lower bounds for the norm pressure in arbitrary dimensions and for the singular value pressure in two dimensions. Using this result we give elementary proofs of the continuity of $\mathbf{M}$ and $\mathfrak{s}$, results which were previously established by D.-J. Feng and P. Shmerkin using multiplicative ergodic theory and thermodynamic formalism. Our methods also yield a new sufficient condition for the continuity of $\mathbf{P}$ which extends a previous result of Feng and Shmerkin; for two-dimensional matrices, our condition is necessary and sufficient. We also extend, unify and simplify results of Y. Guivarc'h, \'{E}. Le Page, M. Ogura and C.F. Martin which characterise the zero-temperature limit of the norm pressure in terms of the joint spectral radius. Finally, our results imply that the norm pressure (or $p$-radius) and affinity dimension of a finite set of matrices can in principle be computed rigorously to any prescribed degree of accuracy.

We will find it convenient to formulate our results in the context not of finite sets of matrices, but of measures. Given a Borel measure $\mu$ on $M_d(\mathbb{R})$ and a real number $s>0$ such that $\int \|A\|^sd\mu(A)<\infty$, let us define
\[\mathbf{M}(\mu,s):=\lim_{n \to \infty} \frac{1}{n}\log \left(\int \left\|A_1\cdots A_n\right\|^s d\mu(A_1)\ldots d\mu(A_n)\right),\]
and if instead $\int \varphi^s(A)d\mu(A)<\infty$ let us define
\[\mathbf{P}(\mu,s):=\lim_{n \to \infty} \frac{1}{n}\log \left(\int \varphi^s\left(A_1\cdots A_n\right) d\mu(A_1)\ldots d\mu(A_n)\right).\]
The pressures $\mathbf{M}((A_1,\ldots,A_N),s)$ and $\mathbf{P}((A_1,\ldots,A_N),s)$ of a finite set of matrices $A_1,\ldots,A_N$ defined previously may easily be seen to correspond to the pressures $\mathbf{M}(\mu,s)$, $\mathbf{P}(\mu,s)$ with $\mu:=\sum_{i=1}^N \delta_{A_i}$. To simplify our notation further, given a measure $\mu$ as above we define for each $n \geq 1$ a measure $\mu_n$ on  $M_d(\mathbb{R})$ by
\[\mu_n(X):=\left(\mu \times \cdots \times \mu\right)\left(\left\{(A_1,\ldots,A_n) \in M_d(\mathbb{R})^n \colon A_1\cdots A_n \in X\right\}\right)\]
for all Borel sets $X\subseteq M_d(\mathbb{R})$, where the product is of $n$ copies of the measure $\mu$. We then have
\[\int f(A)d\mu_n(A)=\int f\left(A_1\cdots A_n\right) d\mu(A_1)\ldots d\mu(A_n)\]
for all measurable functions $f \colon M_d(\mathbb{R}) \to \mathbb{R}$ for which the integrals are well-defined. In particular we have
\begin{equation}\label{eq:naff}\mathbf{M}(\mu,s)=\lim_{n \to \infty} \frac{1}{n}\log \int \|A\|^sd\mu_n(A)=\inf_{n \geq 1} \frac{1}{n}\log \int \|A\|^sd\mu_n(A),\end{equation}
\begin{equation}\label{eq:naffer}\mathbf{P}(\mu,s)=\lim_{n \to \infty} \frac{1}{n}\log \int \varphi^s(A)d\mu_n(A)=\inf_{n \geq 1} \frac{1}{n}\log \int \varphi^s(A)d\mu_n(A)\end{equation}
for all measures $\mu$ on $M_d(\mathbb{R})$ such that $\int \|A\|^sd\mu(A)$, $\int \varphi^s(A)d\mu(A)$ are finite. We note also the elementary identities $\mathbf{M}(\mu_n,s)=n\mathbf{M}(\mu,s)$, $\mathbf{P}(\mu_n,s)=n\mathbf{P}(\mu,s)$ which will be frequently used without comment.

The crux of this article is the following inequality, which is inspired by a theorem of J. Bochi for the joint spectral radius \cite[Theorem A]{Bo03}:
\begin{theorem}\label{th:mpressure}
Let $\mu$ be a measure on $M_d(\mathbb{R})$ and $s>0$ a real number such that $\int \|A\|^s\,d\mu(A)<\infty$. Then
\begin{equation}\label{eq:blahcunt}\int \|A\|^s d\mu_d(A) \leq K_{d,s} e^{\mathbf{M}(\mu,s)}\left(\int \|A\|^sd\mu(A)\right)^{d-1}\end{equation}
where $K_{d,s}:=d^{2+\left(d+1\right)s}\max\left\{d^{1-s},1\right\}$.
\end{theorem}
We observe that Theorem \ref{th:mpressure} implies a simple characterisation of the property $\mathbf{M}(\mu,s)=-\infty$:  if $\mu$ is a measure on $M_d(\mathbb{R})$ such that $\int \|A\|^sd\mu(A)<\infty$, then in view of \eqref{eq:naff} and \eqref{eq:blahcunt}, $\mathbf{M}(\mu,s)=-\infty$ if and only if $\int \|A\|^sd\mu_d(A)=0$.  However, the most signficiant immediate application of this theorem is as follows: whereas the subadditivity of the sequence $\log \int \|A\|^sd\mu_n$ yields the identity
\[\mathbf{M}(\mu,s) = \inf_{n \geq 1}\frac{1}{n}\log\int \|A\|^s d\mu_n(A)\]
as in \eqref{eq:naff}, by applying Theorem \ref{th:mpressure} to each $\mu_n$ we may derive the identity
\[\mathbf{M}(\mu,s) = \sup_{n \geq 1}\frac{1}{n}\log\left(\frac{\int \|A\|^s\mu_{nd}(A)}{K_{d,s}\left(\int \|A\|^s d\mu_n(A)\right)^{d-1}}\right)\]
when $\mathbf{M}(\mu,s)>-\infty$. Acting in concert these upper and lower estimates make certain continuity and limit properties of the norm pressure almost trivial. The applications of Theorem \ref{th:mpressure} are described in detail in the following section.

\section{Applications and extensions of Theorem \ref{th:mpressure}}

\subsection{Continuity properties of matrix pressure functions}

We begin by presenting those applications of Theorem \ref{th:mpressure}  which pertain to the continuity of $\mathbf{M}$, $\mathbf{P}$ and $\mathfrak{s}$. In order to simplify the statement of our continuity results we will  restrict our attention to finite measures $\mu$  whose support is contained in a prescribed compact subset of $M_d(\mathbb{R})$; subject to suitable attention to the problems caused by the possible divergence of integrals, more general results could in principle be derived by similar methods. By rescaling $M_d(\mathbb{R})$ if necessary we lose no generality in considering measures on the closed unit ball of $M_d(\mathbb{R})$, and by rescaling the measures if necessary we shall assume them to be probability measures. We therefore denote the set of all Borel probability measures on the closed unit ball of $M_d(\mathbb{R})$ by $\mathcal{M}_d$ and equip this set with the weak-* topology, which is the smallest topology such that $\mu \mapsto \int f\,d\mu$ is continuous for every real-valued continuous function $f$ defined on the closed unit ball of $M_d(\mathbb{R})$. With respect to this topology $\mathcal{M}_d$ is compact and metrisable.

Our first application of Theorem \ref{th:mpressure} is the following simple result:
\begin{theorem}\label{th:mcontinuous}
The function $\mathbf{M} \colon \mathcal{M}_d\times (0,+\infty) \to [-\infty,+\infty)$ is continuous.
\end{theorem}
This result has the particular corollary that the norm pressure $\mathbf{M}$ (and hence also the $p$-radius $\varrho_p$)  is continuous on finite sets of matrices with fixed cardinality, since by rescaling $M_d(\mathbb{R})$ if necessary we may assume all of these matrices to lie in the open unit ball of $M_d(\mathbb{R})$, and as previously noted the pressures of these finite sets correspond to the pressures of finite sums of Dirac measures. The continuity of $\mathbf{M}$ as a function of a finite set of matrices was previously proved by D.-J. Feng and P. Shmerkin using a combination of the Oseledets multiplicative ergodic theorem and a characterisation of the pressure via a subadditive variational principle \cite[Theorem 1.3]{FeSh14}. As well as generalising this result to the context of measures our proof is more elementary, requiring no use of ergodic theory or thermodynamic formalism.

More significantly, Theorem \ref{th:mpressure} yields the following extension of a theorem of D.-J. Feng and P. Shmerkin \cite{FeSh14}:
\begin{theorem}\label{th:fcontinuous}
The function $\mathbf{P} \colon \mathcal{M}_d \times (0,+\infty)\to [-\infty,+\infty)$ enjoys the following continuity properties:
\begin{enumerate}[(i)]
\item
The function $\mathbf{P}$ is continuous on $\mathcal{M}_d \times (d-1,+\infty)$ and on each of the sets $\mathcal{M}_d \times (k,k+1]$ for integers $k$ such that $0 \leq k<d$.
\item
For each $\mu \in \mathcal{M}_d$, define a measure $\mu^0$ on the closed unit ball of $M_d(\mathbb{R})$ by
\[\mu^0(X):=\mu\left(\left\{A \in X \colon \det A \neq 0\right\}\right)\]
for every Borel measurable set $X$. Let $k$ be an integer such that $0<k<d$. If $\mathbf{P}(\nu,k)=\mathbf{P}(\nu^0,k)$, then the function $(\mu,s) \mapsto \mathbf{P}(\mu,s)$ is continuous at $(\nu,k)$.
\end{enumerate}
\end{theorem}

Similarly to Theorem \ref{th:mcontinuous}, it was previously shown by D.-J. Feng and P. Shmerkin that for fixed $s$, $\mathbf{P}$ is continuous as function of finitely many invertible matrices $A_1,\ldots,A_N \in M_d(\mathbb{R})$, and that $\mathbf{P}$ is continuous in $((A_1,\ldots,A_N),s)$ when $s$ is not one of the integers $1,\ldots,d-1$ or when all of the matrices $A_1,\ldots,A_N$ are invertible \cite[Theorem 1.2]{FeSh14}. These results may be rederived from Theorem \ref{th:fcontinuous} by considering finitely supported measures in the same manner as for $\mathbf{M}$.  As with Theorem \ref{th:mcontinuous}, the work of Feng and Shmerkin relies on the multiplicative ergodic theorem and a variational characterisation of the pressure; our proof, on the other hand, is an application of Theorems \ref{th:mpressure} and \ref{th:mcontinuous}. As noted by Feng and Shmerkin, the above result implies the following:
\begin{theorem}\label{th:cks}
Denote the open unit ball of $M_d(\mathbb{R})$ by $\mathbf{B}_d$. For each $A_1,\ldots,A_N \in \mathbf{B}_d$ define
\[\mathfrak{s}(A_1,\ldots,A_N):=\inf \left\{s>0 \colon \lim_{n \to \infty} \frac{1}{n}\log \sum_{i_1,\ldots,i_n=1}^N \varphi^s\left(A_{i_1}\cdots A_{i_N}\right)<0\right\} \in [0,+\infty).\]
Then $\mathfrak{s} \colon \mathbf{B}_d^N \to [0,+\infty)$ is continuous.
\end{theorem}
In fact, to prove the above statement it is sufficient to know that $\mathbf{P}(\sum_{i=1}^N \delta_{A_i},s)$ depends continuously on $A_1,\ldots,A_N$ for every fixed rational $s$, and this fact can easily be deduced from Theorem \ref{th:mcontinuous}: see \S{\ref{se:cuntcunt}} below.

In two dimensions one may easily derive the following sharper characterisation of the continuity prioperties of $\mathbf{P}$:
\begin{theorem}\label{th:fcontinuous2}
The function $\mathbf{P} \colon \mathcal{M}_2 \times (0,+\infty) \to [-\infty,+\infty)$ is \emph{discontinuous} at $(\mu,s)$ if and only if $s=1$ and $\mathbf{P}(\mu,1)>\mathbf{P}(\mu^0,1)$.
\end{theorem}
Intuitively, Theorem \ref{th:fcontinuous2} asserts that $\mathbf{P}$ is discontinuous at $(\mu,1)$ if and only if the contribution to the pressure $\mathbf{P}(\mu,1)=\mathbf{M}(\mu,1)$ made by non-invertible matrices is negligible. It would be interesting to be able to more simply express the criterion $\mathbf{P}(\mu^0,1)<\mathbf{P}(\mu,1)$, but this problem appears to be somewhat delicate: for example, if $\mu=\delta_{A_1}+\delta_{A_2}$ where $A_1$ is the identity and $A_2$ is a rank-one matrix, then the reader may easily verify that $\mathbf{P}(\mu^0,1)=0$ irrespective of the choice of $A_2$, but $\mathbf{P}(\mu,1)=0$ if and only if $A_2$ is nilpotent.

We lastly note an application of the continuity of $\mathbf{M}$ and $\mathbf{P}$ which could have been obtained from the results in \cite{FeSh14}, but seems to have been previously unremarked:
\begin{proposition}\label{pr:whatwhat}
Let $A_1,\ldots,A_N \in M_d(\mathbb{R})$ and suppose that the matrices $A_i$ are simultaneously block upper-triangular in the following sense: there exist natural numbers $d_1,\ldots,d_k$ with $\sum_{i=1}^kd_i=d$ and matrices $A^{(ij)}_r$ with $1 \leq i \leq j \leq k$, $r=1,\ldots,N$, such that
\[A_r=\left(\begin{array}{ccccc}A^{(11)}_r&A^{(12)}_r&A_r^{(13)}&\cdots &A^{(1k)}\\
0&A^{(22)}_r&A^{(23)}_r&\cdots &A^{(2k)}\\
0&0&A^{(33)}_r&\cdots &A_r^{(3k)}\\
\vdots& \vdots&\vdots  & \ddots & \vdots\\
0 & 0 &0 & \cdots &A^{(kk)}_r
\end{array}\right)\]
for each $r=1,\ldots,N$, where each matrix $A_r^{(ij)}$ is of dimension $d_i\times d_j$. Let $A_1^D,\ldots,A^D_N \in M_d(\mathbb{R})$ be the block-diagonal matrices obtained from $A_1,\ldots,A_N$ by replacing every off-diagonal matrix $A^{(ij)}_r$ with zero. Then
\[\mathbf{M}\left(\sum_{i=1}^N \delta_{A_i},s\right)=\mathbf{M}\left(\sum_{i=1}^N \delta_{A_i^D},s\right),\qquad \mathbf{P}\left(\sum_{i=1}^N \delta_{A_i},s\right)=\mathbf{P}\left(\sum_{i=1}^N \delta_{A_i^D},s\right)\]
for every $s>0$.
\end{proposition}
\begin{proof}
The definitions of $\mathbf{M}$ and $\mathbf{P}$ combined with the elementary inequalities $\|B^{-1}AB\|^s\leq \|B^{-1}\|^s \|A\|^s \|B\|^s$ and $\varphi^s(B^{-1}AB)\leq \|B^{-1}\|^s\varphi^s(A)\|B\|^s$ imply that $\mathbf{M}(\cdot,s)$ and $\mathbf{P}(\cdot,s)$ are constant on the set
\[\left\{\sum_{i=1}^N \delta_{B^{-1}A_iB}\colon B \in GL_d(\mathbb{R})\right\},\]
and by continuity those functions are also constant on the closure of that set: but $\sum_{i=1}^N\delta_{A_i^D}$ belongs to that closure.
\end{proof}
In the special case of the singular value pressure of upper triangular matrices, Proposition \ref{pr:whatwhat} was previously noted by K. Falconer and J. Miao \cite{FaMi07} where it was proved by a long direct calculation; a related two-dimensional result may be found in \cite{Ba08}. The case of block upper-triangular matrices, and the short proof given above, are new.

\subsection{Inequalities for the singular value pressure}\label{ss:one}
Theorem \ref{th:mpressure} has the following partial analogue for the singular value pressure:
\begin{theorem}\label{th:fpressure}
Let $\mu$ be a measure on $M_2(\mathbb{R})$ and $s>0$ a real number such that $\int \varphi^s(A)\,d\mu(A)<\infty$. Then
\[\int \varphi^s(A) d\mu_2(A) \leq \tilde{K}_s e^{\mathbf{P}(\mu,s)}\int \varphi^s(A)d\mu(A)\]
where
\[\tilde{K}_{s}:=\left\{\begin{array}{cl}2^{3+2s}&\text{if }0<s\leq 1,\\
2^{7-2s}&\text{if }1\leq s <2,\\
1&\text{if }s \geq 2.\end{array}\right.\]
\end{theorem}
The methods used in this article do not appear to be sufficient to prove an inequality analogous to \eqref{eq:blahcunt} for $\mathbf{P}$ in the case where $d>2$. We however conjecture the following extension of Theorem \ref{th:fpressure}:
\begin{conjecture}\label{co:only}
For each $d \geq 1$ there exist an integer $n(d)$ and a continuous function $\hat{K}_d \colon (0,+\infty) \to (0,+\infty)$ such that for every measure $\mu$ on $M_d(\mathbb{R})$ for which $\int \varphi^s(A)d\mu(A)<\infty$,
\[\int \varphi^s(A)d\mu_{n(d)}(A) \leq \hat{K}_d(s)e^{\mathbf{P}(\mu,s)}\left(\int \varphi^s(A)d\mu(A)\right)^{n(d)-1}.\]
\end{conjecture}
Using the methods of this article it is possible to achieve a weaker inequality of the form
\begin{equation}\label{eq:badness}\int \varphi^s(A)d\mu_{n(s,d)}(A) \leq M_{s,d}e^{\mathbf{P}(\mu,s)}\left(\int \varphi^s(A)d\mu(A)\right)^{n(s,d)-1}\end{equation}
for rational $s>0$, but the constants $n(s,d)$ and $M_{s,d}$ grow super-exponentially with the denominator of $s$: see Proposition \ref{pr:flight} below. We anticipate that some simplification of the proof of Theorem \ref{th:mcontinuous} might be achievable  via Conjecture \ref{co:only}.

\subsection{Zero temperature limits and the joint spectral radius}

We next turn our attention to the relationship of Theorem \ref{th:mpressure} with earlier work of J. Bochi. Given a compact set $\mathsf{A}\subset M_d(\mathbb{R})$, the joint spectral radius $\varrho_\infty(\mathsf{A})$ is defined to be the quantity
\[\varrho_\infty(\mathsf{A}):=\lim_{n \to \infty} \sup\left\{\left\|A_1\cdots A_n\right\|^{\frac{1}{n}}\colon A_i \in \mathsf{A}\right\}.\]
This limit exists (by subadditivity) and is independent of the choice of norm on $M_d(\mathbb{R})$. In the article \cite{Bo03}, J. Bochi proved that for each $d \geq 1$ there is a constant $\tilde{C}_d>0$ such that for all compact sets $\mathsf{A}\subset M_d(\mathbb{R})$,
\begin{equation}\label{eq:bochi}\sup_{A_1,\ldots,A_d \in \mathsf{A}}\left\|A_1\cdots A_d\right\| \leq \tilde{C}_d\varrho_\infty(\mathsf{A})\left(\sup_{A \in \mathsf{A}}\|A\|\right)^{d-1}.\end{equation}
Theorems \ref{th:mpressure} and \ref{th:fpressure} may be seen as versions of Bochi's inequality for the pressure functions $\mathbf{M}$ and $\mathbf{P}$. Indeed, we may obtain \eqref{eq:bochi} as a limit case of Theorem \ref{th:mpressure} via the following result:
\begin{theorem}\label{th:limit}
Let $\mathsf{A}\subset M_d(\mathbb{R})$ be compact, and let $\mu$ be a finite measure on $M_d(\mathbb{R})$ with support equal to $\mathsf{A}$. Then
\[\lim_{s \to \infty}e^{\mathbf{M}(\mu,s)/s} = \varrho_\infty(\mathsf{A}).\]
\end{theorem}
If $\mathsf{A}\subset M_d(\mathbb{R})$ is compact and $\mu$ is a probability measure with support equal to  $\mathsf{A}$, then taking the power $\frac{1}{s}$ in \eqref{eq:blahcunt} and applying Theorem \ref{th:mpressure} we trivially recover Bochi's inequality \eqref{eq:bochi} with $\tilde{C}_d:=d^{d+1}$. We remark that a version of Theorem \ref{th:limit} was previously obtained by Y. Guivarc'h and \'{E}. Le Page in the case where $\mu$ is supported on the set of invertible matrices, and where a strong irreducibility condition is satisfied by the semigroup generated by the support of $\mu$, as part of a general operator-theoretic investigation of Lyapunov exponents under those assumptions: see \cite[Th\'eor\`eme 4.17]{GuLe04}. A special case of Theorem \ref{th:limit} has also been given by M. Ogura and C. F. Martin under positivity and non-singularity assumptions on the support of the measure \cite{OgMa14}.  Our proof is an elementary application of Theorem \ref{th:mpressure} and makes no assumptions of reducibility, invertibility, non-singularity or positivity. Results of this type may be viewed as an extension to matrices of the ``zero-temperature limit'' in thermodynamic formalism: for some results in the $1$-dimensional case we note for example \cite{BaLeLo12,Br03,ChHo10}.

\subsection{Implications for the computation of the pressure}\label{ss:two}
The results of this article have some theoretical implications for the computation of  $\mathbf{M}$ and $\mathbf{P}$. Theorem \ref{th:mpressure} allows us to show that the algorithmic computation of $\mathbf{M}$, while potentially difficult, is easier than the computation of Lyapunov exponents in the following precise sense: there exists an algorithm which is guaranteed to compute the norm pressure of a finite set of matrices $A_1,\ldots,A_N$ at a given parameter $s>0$ to within a specified precision $\varepsilon>0$ in a finite amount of time. In the context of the $p$-radius (which is equivalent to the norm pressure via \eqref{eq:pradius}) a need for such algorithms was recently highlighted by R. Jungers and V. Protasov (\cite[\S6]{JuPr11} and \cite{Pr10}). Indeed, V. Protasov has shown in \cite{Pr10} that a finite set of matrices $A_1,\ldots,A_N\in M_d(\mathbb{R})$ preserves a cone in $\mathbb{R}^d$ if and only if $\varrho_1(A_1,\ldots,A_N)=\rho(\sum_{i=1}^N A_i)$, and this provides further motivation for the problem of computing $\mathbf{M}$.

Given $\varepsilon,s>0$ and $A_1,\ldots,A_N \in M_d(\mathbb{R})$, let $\mu:=\sum_{i=1}^N\delta_{A_i}$. By Theorem \ref{th:mpressure} the pressure $\mathbf{M}(\mu,s)$ is equal to $-\infty$ if and only if $\int \|A\|^sd\mu_d(A)=0$, which is to say if and only if all of the products $A_{i_1}\cdots A_{i_d}$ with $1 \leq i_1,\ldots,i_d \leq N$ are equal to the zero matrix. If this is not the case, then we observe that by Theorem \ref{th:mpressure} we have for all $n \geq 1$
\[\left(\frac{\int \|A\|^s d\mu_{nd}(A)}{K_{s,d}\left(\int \|A\|^sd\mu_n(A)\right)^{d-1}}\right)^{\frac{1}{n}} \leq e^{\mathbf{M}(\mu,s)}\leq \left( \int\|A\|^sd\mu_n(A)\right)^{\frac{1}{n}}\]
and as $n \to \infty$ both the left-hand and right-hand terms converge to the middle term. It follows that if we compute these expressions by brute force for an increasing sequence of values of $n$, we must after a finite amount of computation necessarily arrive at an integer $n$ such that
\[\frac{1}{n}\log \int \|A\|^sd\mu_n(A) -\frac{1}{n}\log\left(\frac{\int \|A\|^s d\mu_{nd}(A)}{K_{s,d}\left(\int \|A\|^sd\mu_n(A)\right)^{d-1}}\right) < \varepsilon\]
and for this integer $n$ we have
\[\left|\mathbf{M}(\mu,s)-\frac{1}{n}\log \int \|A\|^sd\mu_n(A)\right|<\varepsilon\]
as desired. 
In view of Theorem \ref{th:fpressure} it is clear that an analogous algorithm exists for the estimation of the singular value pressure $\mathbf{P}(\mu,s)$ in two dimensions for all values of $s$, and using \eqref{eq:badness} this extends to the computation of $\mathbf{P}(\mu,s)$ in arbitrary dimensions when $s>0$ is rational. Since $s \mapsto \mathbf{P}(\mu,s)$ is monotone, this implies that $\mathbf{P}(\mu,s)$ can be computed (perhaps laboriously) for any computable value of $s$ by taking successive upper and lower rational approximations to $s$.  If Conjecture \ref{co:only} is valid then the algorithm for the approximation of $\mathbf{M}(\mu,s)$ extends directly to the computation of $\mathbf{P}(\mu,s)$ in all dimensions.  The efficiency of these approximation procedures seems likely to be low, but this situation nonetheless contrasts strongly with that for the top Lyapunov exponent of the matrices $A_1,\ldots,A_N$,
\[\Lambda(A_1,\ldots,A_N):=\lim_{n \to \infty} \frac{1}{N^n} \sum_{i_1,\ldots,i_n=1}^N \log \left\|A_{i_1}\cdots A_{i_n}\right\|,\]
since it is known that an algorithm which is guaranteed to always estimate the Lyapunov exponent $\Lambda(A_1,\ldots,A_N)$ to within a prescribed error of $\varepsilon$ in a finite number of steps cannot exist, even if $N=2$ and the matrices are assumed to have integer entries \cite[Theorem 2]{TsBl97}. The successful development of efficient algorithms for the joint spectral radius, and in some cases the $p$-radius -- see for example \cite{BlNe05,GuPr13,JuPr11,OgPrJu15,PrJuBl09} -- suggests that more efficient algorithms for the computation of $\mathbf{M}$ and $\mathbf{P}$ are likely to exist.

\subsection{Implications for the computation of the affinity dimension}
The results in this article also yield a procedure for the rigorous estimation of the affinity dimension $\mathfrak{s}(A_1,\ldots,A_N)$ to within a prescribed error $\varepsilon>0$ when $A_1,\ldots,A_N \in M_d(\mathbb{R})$ have norm strictly less than one. The algorithm which we will present seems likely to be too inefficient for reasonable practical use, but we nonetheless consider it interesting since the existence of a general algorithm of this kind does not seem to have been previously suspected. We observe firstly that $\mathfrak{s}(A_1,\ldots,A_N) \geq d$ if and only if
\[\mathbf{P}(\mu,d) =\log \int  |\det A|d\mu(A) \geq 0,\]
and in this case $\mathfrak{s}(A_1,\ldots,A_N)$ is given by the unique value of $s$ which solves
\[\mathbf{P}(\mu,s)=\log\int  |\det A|^{s/d}d\mu(A)=0,\]
an equation which may be solved by elementary methods. The interesting cases of the problem therefore occur only when $\mathfrak{s}(A_1,\ldots,A_N) \in [0,d]$, and we shall assume henceforth that this is known to hold. To obtain the desired algorithm it is sufficient to be able to solve the following problem in a finite number of steps: given that $\mathfrak{s}(A_1,\ldots,A_N)$ belongs to a closed interval $I$ with rational endpoints, find a subinterval $J$ of $I$ which contains $\mathfrak{s}(A_1,\ldots,A_N)$, has rational endpoints, and has length at most two thirds that of $I$. By applying this procedure iteratively starting with the interval $[0,d]$ it is clear that we may obtain an estimate of $\mathfrak{s}(A_1,\ldots,A_N)$ to within an \emph{a priori} prescribed accuracy in finitely many steps.

As before let us for convenience write $\mu:=\sum_{i=1}^N\delta_{A_i}$. The function $s \mapsto \mathbf{P}(\mu,s)$ may easily be shown to be strictly decreasing except possibly on an interval of the form $(k,+\infty)$ where it takes the value $-\infty$, see \S\ref{se:cuntcunt} below. Given the interval $I=[s_1,s_2]$ let $t_1:=\frac{2}{3}s_1+\frac{1}{3}s_2$ and $t_2:=\frac{1}{3}s_1+\frac{2}{3}s_2$.  Using \eqref{eq:badness} there exist easily-computable constants $d_1,d_2,M_1,M_2 \gg 1$ depending only on the numerator, denominator and integer parts of $t_1$ and $t_2$ and on $d$ such that for all $n \geq 1$
\[\left(\frac{\int \varphi^{t_i}(A) d\mu_{nd_i}(A)}{M_i\left(\int \varphi^{t_i}(A)d\mu_n(A)\right)^{d_i-1}}\right)^{\frac{1}{n}} \leq e^{\mathbf{P}(\mu,{t_i})}\leq \left( \int\varphi^{t_i}(A)d\mu_n(A)\right)^{\frac{1}{n}}\]
for $i=1,2$, and furthermore both the left and right-hand terms converge to the middle term in the limit as $n \to \infty$. By the strictly decreasing property of $\mathbf{P}$ the two values $\mathbf{P}(\mu,t_1)$ and $\mathbf{P}(\mu,t_2)$ are not both equal to zero, and therefore there exists $i \in\{1,2\}$ such that either
\begin{equation}\label{eq:more} \int\varphi^{t_i}(A)d\mu_n(A)<1\end{equation}
for all sufficiently large $n$ (by the convergence of the upper estimate) or 
\begin{equation}\label{eq:evenmore}\frac{\int \varphi^{t_i}(A) d\mu_{nd_i}(A)}{M_i\left(\int \varphi^{t_i}(A)d\mu_n(A)\right)^{d_i-1}}>1\end{equation}
for all sufficiently large $n$ (by the convergence of the lower estimate). If the former holds for at least one integer $n$ then necessarily $\mathbf{P}(\mu,t_i)<0$ and therefore $\mathfrak{s}(A_1,\ldots,A_N) \in J:=[s_1,t_i]$; if the latter holds for some integer $n$ then $\mathbf{P}(\mu,t_i)>0$ and therefore $\mathfrak{s}(A_1,\ldots,A_N) \in J:=[t_i,s_2]$. Clearly, by brute-force computation of the expressions $\int\varphi^{t_i}(A)d\mu_n(A)$ and $\int \varphi^{t_i}(A) d\mu_{nd_i}(A)/M_i\left(\int \varphi^{t_i}(A)d\mu_n(A)\right)^{d_i-1}$ for $i=1,2$ for increasing values of $n$ we may after a finite number of computations arrive at an integer $n \geq 1$ and an integer $i \in \{1,2\}$ such that one of the inequalities \eqref{eq:more}, \eqref{eq:evenmore} is satisfied, and once such an inequality is obtained the desired interval $J$ is known. This completes the description of the algorithm.

We remark that while the above observation is encouraging, it may yet be the case that problems such as determining whether or not a given set of rational matrices $A_1,\ldots,A_N$ satisfies $\mathfrak{s}(A_1,\ldots,A_N)\geq \frac{1}{2}$ are algorithmically undecidable: indeed, that particular problem is equivalent to that of determining whether or not the quantitiy $\mathbf{P}(\sum_{i=1}^N \delta_{A_i},\frac{1}{2})=\mathbf{M}(\sum_{i=1}^N \delta_{A_i},\frac{1}{2})=\log(N \varrho_{1/2}(A_1,\ldots,A_N)^{1/2})$ is strictly positive, and the closely-related problem of determining whether or not $\varrho_\infty(\{A_1,\ldots,A_N\})>1$ is known to be algorithmically undecidable \cite{BlTs00}.

\subsection{Overview of the remainder of the article}

The proof of Theorem \ref{th:mpressure} is related to the proof of Bochi's inequality \eqref{eq:bochi}, which may be regarded as consisting of two essential steps: firstly, one shows that for any nonempty compact set $\mathsf{A}\subset M_d(\mathbb{R})$,
\begin{equation}\label{eq:b1}\inf_{B \in GL_d(\mathbb{R})} \sup_{A \in \mathsf{A}}\left\|BAB^{-1}\right\| \leq C \varrho_\infty(\mathsf{A}) \end{equation}
for some constant $C$ depending only on $d$; and secondly, one shows that for any product $A_1\cdots A_d$ of $d \times d$ matrices and any $B \in GL_d(\mathbb{R})$, 
\begin{equation}\label{eq:b2}\|A_1\cdots A_d\| \leq C \max_{1 \leq i \leq d}\left(\left\|A_i\right\| \cdot  \prod_{\substack{ 1 \leq k \leq d\\k \neq i}} \left\|BA_kB^{-1}\right\|\right)\end{equation}
for a second constant $C$ depending only on $d$. The inequality \eqref{eq:bochi} follows by the combination of these two results. 
The proof of Theorem \ref{th:mpressure} admits the same structure as that of Bochi's inequality: in the first part we prove an analogue of \eqref{eq:b1} for the norm pressure $\mathbf{M}$, and the second part adapts the inequality \eqref{eq:b2}. The principal effort of this proof lies in the adaptation of \eqref{eq:b1} to the norm pressure $\mathbf{M}$, which is given in Proposition \ref{pr:gldr-bound} below; the analogue of \eqref{eq:b2} is relatively straightforward. In \S\ref{se:two} we give the proof of Theorem \ref{th:mpressure}, which is the longest proof in this paper. Sections \ref{se:three} through \ref{se:last} provide respectively the proofs of Theorems \ref{th:mcontinuous}--\ref{th:limit}. The proof of Theorem \ref{th:fcontinuous} is somewhat involved, but the remaining proofs are all extremely brief.


\section{Proof of Theorem \ref{th:mpressure}}\label{se:two}
As was indicated previously we begin by proving a lower bound for $\mathbf{M}(\mu,s)$ in terms of distorted Euclidean norms:
\begin{proposition}\label{pr:gldr-bound}
Let $s>0$ and let $\mu $ be a Borel probability measure on $M_d(\mathbb{R})$ such that $\int \|A\|^sd\mu(A)<\infty$. Then
\begin{equation}\label{eq:prp}\inf_{B \in GL_d(\mathbb{R})} \int  \left\|BAB^{-1}\right\|^sd\mu(A) \leq C_{d,s}e^{\mathbf{M}(\mu,s)}\end{equation}
where $C_{d,s}:=\max\{d^2,d^{1+s}\}$.
\end{proposition}
A number of lemmas will be required in order for us to prove this result. The proof of Proposition \ref{pr:gldr-bound} splits into two similar but not identical cases according to whether or not $s$ is strictly less than $1$. The case $s \geq 1$ involves the construction of a family of special norms on $\mathbb{R}^d$ associated to the measure $\mu$ and parameter $s$. The case $0<s<1$ is similar, but requires an object slightly different to a norm:
\begin{definition}\label{de:fnorm}
We define an \emph{$s$-homogenous $F$-norm on $\mathbb{R}^d$}, where $0 <s <1$, to be a function $|\cdot| \colon \mathbb{R}^d \to \mathbb{R}$ such that:
\begin{enumerate}[(i)]
\item
$|u|\geq 0$ for all $u \in \mathbb{R}^d$, with $|u|=0$ if and only if $u=0$
\item
$|u+v|\leq |u|+|v|$ for all $u,v \in\mathbb{R}^d$
\item
$|\lambda u|=|\lambda|^s|u|$ for all $u \in \mathbb{R}^d$ and $\lambda \in \mathbb{R}$.
\end{enumerate}
\end{definition}
For $0<s<1$ the function $|(x_1,\ldots,x_d)|:=\sum_{i=1}^d|x_i|^s$ can be seen to be an $s$-homogenous $F$-norm by virtue of the inequality $(x+y)^s \leq x^s+y^s$ which is valid for all $x,y \geq 0$ and $s \in (0,1)$. 
We observe that an $s$-homogenous $F$-norm $|\cdot|$ is $s$-H\"older continuous with respect to the usual distance on $\mathbb{R}^d$: if $e_1,\ldots,e_d$ denotes the standard basis, then
\[\left|\left|u+\sum_{i=1}^d \lambda_i e_i\right| -\left|u\right|\right| \leq \left|\sum_{i=1}^d\lambda_i e_i\right| \leq \sum_{i=1}^d \left|\lambda_i\right|^s|e_i| = O\left(\left\|\sum_{i=1}^d \lambda_i e_i\right\|^s\right).\]
In particular the set $\{v \in \mathbb{R}^d \colon |v|\leq 1\}$ is closed, and by (i) it contains an open neighbourhood of the origin; however, it will in general not be convex. 

 In the case where $s \geq 1$ and $\mu$ is the sum of finitely many Dirac measures the following lemma was previously given by Cabrelli, Heil, and Molter \cite[Proposition 2.17]{CaHeMo04}. If additionally the matrices in the support of $\mu$ do not have a common invariant subspace then the conclusion holds with $\varepsilon=0$, see \cite{Pr08}; under stronger irreducibility conditions and when $\mu$ is supported on the set of invertible matrices, we note the related result \cite[Th\'eor\`eme 4.1]{GuLe04}. The analogue of Lemma \ref{le:exists-norm} for the joint spectral radius dates back to 1960 (\cite{RoSt60}, reprinted in \cite{Ro03}).
\begin{lemma}\label{le:exists-norm}
Let $\mu$ be a measure on $M_d(\mathbb{R})$, and let $s>0$ such that $\int \|A\|^sd\mu(A)<\infty$. Let $\varepsilon>0$. If $0<s<1$, then there exists an $s$-homogenous $F$-norm $|\cdot|_\varepsilon$ on $\mathbb{R}^d$ such that for all $v \in \mathbb{R}^d$
\begin{equation}\label{eq:extr}\int |Av|_\varepsilon d\mu(A) \leq \left(e^{\mathbf{M}(\mu,s)}+\varepsilon\right)|v|_\varepsilon.\end{equation}
If $s \geq 1$, then there exists a norm $\|\cdot\|_\varepsilon$ on $\mathbb{R}^d$ such that for all $v \in \mathbb{R}^d$
\begin{equation}\label{eq:extr2}\int \|Av\|_\varepsilon^sd\mu(A) \leq \left(e^{\mathbf{M}(\mu,s)}+\varepsilon\right)\|v\|_\varepsilon^s.\end{equation}
\end{lemma}
\begin{proof}
Suppose first that $0<s<1$.  For each $v \in \mathbb{R}^d$ let us define
\[|v|_\varepsilon:=\|v\|^s+\sum_{n =1}^\infty \left(e^{\mathbf{M}(\mu,s)}+\varepsilon\right)^{-n} \int \|A\|^sd\mu_n(A).\]
The convergence of the series is immediate from the definition of $\mathbf{M}(\mu,s)$, and the fact that $|\cdot|_\varepsilon$ is an $s$-homogenous $F$-norm follows from the fact that $\|\cdot\|$ is a norm and from the elementary inequality $|x+y|^s \leq |x|^s+|y|^s$. The verification of \eqref{eq:extr} is a simple calculation. In the case $s \geq 1$ we instead define
\[\|v\|_\varepsilon:=\left(\|v\|^s+\sum_{n =1}^\infty \left(e^{\mathbf{M}(\mu,s)}+\varepsilon\right)^{-n} \int \|Av\|^sd\mu_n(A)\right)^{\frac{1}{s}}.\]
The convergence of the series is again immediate. The triangle inequality for ${\|\cdot\|_\varepsilon}$ follows from the triangle inequalities for $\|\cdot\|$, $L^s(\mu_n)$ and $\ell^s(\mathbb{N})$. The homogeneity of $\|\cdot\|_\varepsilon$ and the inequality \eqref{eq:extr2} are obvious.
\end{proof}
In the case where $0<s<1$, the following geometric lemma allows us to relate the unit ball of an $s$-homogenous $F$-norm to that of a linearly distorted Euclidean norm:
\begin{lemma}\label{le:exists-b}
Let $0<s<1$ and suppose that $|\cdot|$ is an $s$-homogenous $F$-norm on $\mathbb{R}^d$. Then there exists $B \in GL_d(\mathbb{R})$ such that $d^{s-1} |v|\leq \|Bv\|^s \leq d^{\frac{s}{2}} |v|$ for all $v \in \mathbb{R}^d$. 
\end{lemma}
\begin{proof}
Let $X:=\{v \in \mathbb{R}^d \colon |v|\leq 1\}$. Since $|\cdot|$ is an $s$-homogenous $F$-norm it is continuous with respect to the usual topology on $\mathbb{R}^d$, and therefore $X$ is closed (see remarks following Definition \ref{de:fnorm}). Since $|v|=0$ if and only if $v=0$ it follows easily that $X$ contains an open neighbourhood of the origin. 

Let $Z$ denote the convex hull of $X$, which is closed. Carath\'eodory's theorem on convex hulls (see e.g. \cite[Theorem 1.1.4]{Sc14}) asserts that the convex hull of a subset $Y$ of $\mathbb{R}^d$ is precisely the set of all convex combinations of subsets of $Y$ having cardinality $d+1$. In the case where $Y$ is connected, an observation of W. Fenchel shows that the required cardinality may be reduced to $d$: see \cite{Ba12,Fe29}. Clearly every point of $X$ is connected to the origin by a straight line, so $X$ is a connected set and therefore
\[Z=\left\{\sum_{i=1}^{d}\lambda_i v_i \colon v_1,\ldots,v_{d} \in X,\text{ }\lambda_1,\ldots,\lambda_{d} \geq 0\text{ and }\sum_{i=1}^{d}\lambda_i\leq 1\right\}.\]
If $v \in Z$, let us write $v=\sum_{i=1}^{d}\lambda_i v_i$ in the above fashion. We have
\[\left|v\right| \leq \sum_{i=1}^{d}\left|\lambda_iv_i\right| \leq \sum_{i=1}^{d}\lambda_i^s\leq d^{1-s}\]
and it follows that $X \subseteq Z \subseteq d^{\frac{1}{s}-1}X$. Since $Z$ has nonempty interior and is convex and symmetrical with respect to the map $v \mapsto -v$, John's theorem (see e.g. \cite[Theorem 10.12.2]{Sc14}) implies that there exists a closed ellipsoid $E\subseteq Z$ centred at the origin such that $Z\subseteq \sqrt{d}\cdot E$. Let $B\in GL_d(\mathbb{R})$ be a matrix which maps the ellipsoid $E$ bijectively onto the unit ball of $\mathbb{R}^d$. If $\|Bv\|^s=1$ then $v \in E \subseteq Z \subseteq d^{\frac{1}{s}-1}X$ and therefore $|v| \leq d^{1-s}=d^{1-s}\|Bv\|^s$. On the other hand if $|v|=1$ then $v \in X \subseteq Z \subseteq \sqrt{d}\cdot E$ and therefore $\|Bv\|^s \leq  d^{s/2}=d^{s/2}|v|$. The result for general $v$ follows by $s$-homogeneity.
\end{proof}
The following simpler version of Lemma \ref{le:exists-b} will be used to treat the case $s \geq 1$:
\begin{lemma}\label{le:exists-c}
Suppose that $\|\cdot\|_*$ is a norm on $\mathbb{R}^d$. Then there exists $B \in GL_d(\mathbb{R})$ such that $\|v\|_*\leq \|Bv\| \leq d^{\frac{1}{2}} \|v\|_*$ for all $v \in \mathbb{R}^d$. 
\end{lemma}
\begin{proof}
Let $X$ denote the convex set $\{v \in \mathbb{R}^d \colon \|v\|_*\leq 1\}$. By John's theorem there exists an ellipsoid $E$ centred at the origin of $\mathbb{R}^d$ such that $E \subseteq X\subseteq \sqrt{d}\cdot E$. Let $B \in GL_d(\mathbb{R})$ be a matrix which maps $E$ bijectively onto the Euclidean unit ball. If $\|v\|_*=1$ then $v \in X \subseteq \sqrt{d}\cdot E$ and therefore $\|Bv\|\leq \sqrt{d}=d^{\frac{1}{2}}\|v\|_*$; conversely if $\|Bv\|=1$ then $v \in E \subseteq X$ and therefore $\|v\|_* \leq 1=\|Bv\|$.
\end{proof}

Finally we note the following elementary lemma:
\begin{lemma}\label{le:sup-inside}
Let $A \in M_d(\mathbb{R})$ and let $e_1,\ldots,e_d$ be an orthonormal basis for $\mathbb{R}^d$. Then
\[\max_{1 \leq i \leq d} \|Ae_i\| \geq d^{-\frac{1}{2}}\|A\|.\]
\end{lemma}
\begin{proof}
By compactness we may choose $v \in \mathbb{R}^d$ such that $\|v\|=1$ and $\|Av\|=\|A\|$. Choose $\lambda_1,\ldots,\lambda_d \in \mathbb{R}$ such that $v=\sum_{i=1}^d \lambda_i e_i$. Since $\sum_{i=1}^d |\lambda_i|^2=1$ we have $\sum_{i=1}^d |\lambda_i|\leq \sqrt{d}$ and therefore
\[\left\|A\right\|=\left\|Av\right\| \leq \sum_{i=1}^d \left|\lambda_i\right| \cdot\left\|Ae_i\right\| \leq \sqrt{d}\cdot\max_{1 \leq i \leq d}\left\|Ae_i\right\|\]
as required.\end{proof}

\begin{proof}[Proof of Proposition \ref{pr:gldr-bound}]
Let $s,\varepsilon>0$. If $0<s<1$ then by Lemma \ref{le:exists-norm} there exists an $s$-homogenous $F$-norm $|\cdot|_\varepsilon$ on $\mathbb{R}^d$ such that for all $v \in \mathbb{R}^d$
\begin{equation}\label{eq:temp}\int |Av|_\varepsilon d\mu(A)\leq \left(e^{\mathbf{M}(\mu,s)}+\varepsilon\right)|v|_\varepsilon.\end{equation}
Let $B$ be the matrix provided by Lemma \ref{le:exists-b}, which satisfies
\begin{equation}\label{eq:temp2}d^{s-1}|v|_\varepsilon\leq \|Bv\|^s \leq d^{\frac{s}{2}}|v|_\varepsilon\end{equation}
for every $v \in \mathbb{R}^d$.
  By Lemma \ref{le:sup-inside}, 
\begin{align*}\int \left\|BAB^{-1}\right\|^s d\mu(A) &\leq d^{\frac{s}{2}} \int \max_{1 \leq i \leq d}\left\|BAB^{-1}e_i\right\|^s d\mu(A) \\
&\leq d^{\frac{s}{2}}\int \sum_{i=1}^d\left\|BAB^{-1}e_i\right\|^s d\mu(A) \\
&= d^{\frac{s}{2}} \sum_{i=1}^d\int \left\|BAB^{-1}e_i\right\|^s d\mu(A) \\
&\leq d^{1+\frac{s}{2}} \max_{1 \leq i \leq d}\int \left\|BAB^{-1}e_i\right\|^s d\mu(A),\end{align*}
so in particular there exists a unit vector $v \in \mathbb{R}^d$ with $\|v\|=1$ such that
\[\int \left\|BAB^{-1}\right\|^s d\mu(A) \leq d^{1+\frac{s}{2}}\int \left\|BAB^{-1}v\right\|^s d\mu(A).\]
Applying \eqref{eq:temp} and \eqref{eq:temp2} we find that
\begin{align*}\int \left\|BAB^{-1}v\right\|^s d\mu(A) &\leq d^{\frac{s}{2}} \int \left|AB^{-1}v\right|_\varepsilon d\mu(A)\\
&\leq  d^{\frac{s}{2}} \left(e^{\mathbf{M}(\mu,s)}+\varepsilon\right)\left|B^{-1}v\right|_\varepsilon\\
&\leq  d^{1-\frac{s}{2}}\left(e^{\mathbf{M}(\mu,s)}+\varepsilon\right)\|v\|^s,\end{align*}
and by combining the previous two inequalities we obtain
\[\int \left\|BAB^{-1}\right\|^sd\mu(A) \leq d^2\left(e^{\mathbf{M}(\mu,s)}+\varepsilon\right)\]
since $\|v\|=1$. Since $\varepsilon>0$ was arbitrary the conclusion of Proposition \ref{pr:gldr-bound} follows.  The case $s \geq 1$ may be derived by combining Lemmas \ref{le:exists-norm}, \ref{le:exists-c} and \ref{le:sup-inside} in a directly analogous manner.
\end{proof}
Now that Proposition \ref{pr:gldr-bound} has been proved we are close to being able to prove Theorem \ref{th:mpressure}. We require just one further lemma, which the reader should compare with \cite[Lemma 2]{Bo03}. The optimisation of the constant in this lemma is an interesting problem which we do not attempt to address here.
\begin{lemma}\label{le:bochipart2}
Let $A_1,\ldots,A_d \in M_d(\mathbb{R})$ and $B \in GL_d(\mathbb{R})$. Then
\begin{equation}\label{eq:beans} \left\|A_1\cdots A_d\right\| \leq d^{d} \max_{1 \leq k \leq d} \left( \left\|BA_k B^{-1}\right\| \cdot \prod_{\substack{1 \leq i \leq d \\ i \neq k}} \left\|A_i\right\|\right).\end{equation}
\end{lemma}
\begin{proof}
We first claim that it is sufficient to prove \eqref{eq:beans} under the additional assumption that  $B$ is a diagonal matrix. Indeed, let us suppose the lemma to be valid in all cases where the matrix $B$ is diagonal. Given a general matrix $B \in GL_d(\mathbb{R})$, by singular value decomposition we may write $B=UDV$ where $U$ and $V$ are orthogonal matrices and $D$ is an invertible diagonal matrix. Since \eqref{eq:beans} is assumed to be valid for the matrices $VA_1V^{-1}$, $VA_2V^{-1}$, \ldots, $VA_dV^{-1}$ and the diagonal matrix $D$ we have
\[\left\|VA_1\cdots A_dV^{-1}\right\| \leq d^{d} \max_{1 \leq k \leq d} \left( \left\|DVA_kV^{-1} D^{-1}\right\| \cdot \prod_{\substack{1 \leq i \leq d \\ i \neq k}} \left\|VA_iV^{-1}\right\|\right)\]
or, since $U$ and $V$ are isometries with respect to the Euclidean norm $\|\cdot\|$,
\[\left\|A_1\cdots A_d\right\| \leq d^{d} \max_{1 \leq k \leq d} \left( \left\|UDVA_kV^{-1} D^{-1}U^{-1}\right\| \cdot \prod_{\substack{1 \leq i \leq d \\ i \neq k}} \left\|A_i\right\|\right)\]
which is precisely \eqref{eq:beans} for the general matrix $B$. For the remainder of the proof we therefore make the additional hypothesis that $B$ is diagonal.

Let us write $\|A\|_\infty$ for the maximum of the absolute values of the entries of the matrix $A$. Since $\|A\|_\infty \leq \|A\| \leq d\|A\|_\infty$, in order to prove \eqref{eq:beans} it is sufficient for us to prove the inequality 
\[ \left\|A_1\cdots A_d\right\|_\infty \leq d^{d-1} \max_{1 \leq k \leq d} \left( \left\|BA_k B^{-1}\right\|_\infty \cdot \prod_{\substack{1 \leq i \leq d \\ i \neq k}} \left\|A_i\right\|_\infty\right)\]
where $B$ is assumed to be diagonal. 

Let us therefore let $\lambda_1,\ldots,\lambda_d$ denote the diagonal entries of the matrix $B$, and let $a^{(k)}_{ij}$ denote the entry in the $i^{\mathrm{th}}$ row and $j^{\mathrm{th}}$ column of the matrix $A_k$.   We may estimate
\begin{align*}\left\|A_1\cdots A_d\right\|_\infty &= \max_{1 \leq i_0,i_d \leq d} \left|\sum_{i_1,\ldots,i_{d-1}=1}^d a_{i_0 i_1}^{(1)}a_{i_1i_2}^{(2)}\cdots a_{i_{d-1}i_d}^{(d)}\right|\\
&\leq d^{d-1}\max_{1 \leq i_0,i_1,\ldots,i_d \leq d} \left|  a_{i_0 i_1}^{(1)}a_{i_1i_2}^{(2)}\cdots a_{i_{d-1}i_d}^{(d)}\right|\\
&=d^{d-1} \max_{1 \leq i_0,i_1,\ldots,i_d \leq d}\prod_{k=1}^{d} \left|a_{i_{k-1}i_{k}}^{(k)}\right|.\end{align*}
Let us fix a choice of $i_0,\ldots,i_d$ which achieves this maximum. Since the $d+1$ numbers $\lambda_{i_0},\ldots,\lambda_{i_d}$ take at most $d$ distinct values it is impossible to have the chain of inequalities $|\lambda_{i_0}|<|\lambda_{i_1}|<\cdots <|\lambda_{i_d}|$, so there must necessarily exist $\ell\in \{1,\ldots,d\}$ such that $|\lambda_{i_{\ell-1}}|\geq|\lambda_{i_\ell}|$. Thus
\begin{align*}\left\|A_1\cdots A_d\right\|_\infty &\leq d^{d-1}\left|\lambda_{i_{\ell-1}}a_{i_{\ell-1} i_{\ell}}^{(\ell)}\lambda_{i_{\ell}}^{-1
}\right|\prod_{\substack{1 \leq k \leq d\\k \neq \ell}} \left|a_{i_{k-1}i_{k}}^{(k)}\right|\\
&\leq d^{d-1} \left\|BA_\ell B^{-1}\right\|_\infty \prod_{\substack{1 \leq k \leq d\\k \neq \ell}} \left\|A_k\right\|_\infty\\
&\leq d^{d-1} \max_{1 \leq r \leq d}\left(\left\|BA_rB^{-1}\right\|_\infty \prod_{\substack{1 \leq k \leq d\\k \neq r}} \left\|A_k\right\|_\infty\right)\end{align*}
using the fact that $B$ is diagonal, and this is exactly the result required.
\end{proof}
\begin{proof}[Proof of Theorem \ref{th:mpressure}]
Let $\varepsilon>0$. By Proposition \ref{pr:gldr-bound} we may choose $B \in GL_d(\mathbb{R})$ such that 
\[\int  \left\|BAB^{-1}\right\|^sd\mu(A) \leq C_{d,s}\left(e^{\mathbf{M}(\mu,s)}+\varepsilon\right).\]
If $A_1,\ldots,A_d \in M_d(\mathbb{R})$ are arbitrary matrices, then it follows from Lemma \ref{le:bochipart2} that
\[\left\|A_1\ldots A_d\right\|^s \leq d^{sd} \sum_{i=1}^d \left\|BA_iB^{-1}\right\|^s\left(\prod_{\substack{1 \leq k \leq d\\k \neq i}} \|A_k\|^s\right).\]
By integration it follows that
\begin{align*}\int \|A\|^s d\mu_d(A) &= \int \left\|A_1\ldots A_d\right\|^s  d\mu(A_1)d\mu(A_2)\cdots d\mu(A_d)\\
&\leq d^{1+sd} \left(\int\left\|BAB^{-1}\right\|^s d\mu(A)\right)\left(\int \|A\|^sd\mu(A)\right)^{d-1}\\
&\leq d^{1+sd}C_{s,d}\left(e^{\mathbf{M}(\mu,s)}+\varepsilon\right)\left(\int \|A\|^sd\mu(A)\right)^{d-1}\end{align*}
and since $\varepsilon>0$ was arbitrary the result follows.
\end{proof}


\section{Proof of Theorem \ref{th:mcontinuous}}\label{se:three}
For all $\nu \in \mathcal{M}_d$ and $t>0$ we have
\[\mathbf{M}(\nu,t) =\inf_{n \geq 1} \frac{1}{n}\log \left(\int \|A\|^td\nu_n(A)\right)\]
by subadditivity, and this shows that $\mathbf{M}$ is an infimum of continuous functions $\mathcal{M}_d \times (0,+\infty) \to [-\infty,+\infty)$. In particular it is upper semi-continuous. It follows in particular that if $\mathbf{M}(\mu,s)=-\infty$ then $\mathbf{M}$ is continuous at $(\mu,s)$.

If $(\mu,s) \in \mathcal{M}_d\times(0,+\infty)$ and $\mathbf{M}(\mu,s)>-\infty$ then necessarily $\int \|A\|^sd\mu_d(A)>0$, so in particular $\int \|A\|^td\nu_d(A)>0$ for all $(\nu,t)$ sufficiently close to $(\mu,s)$. By Theorem \ref{th:mpressure} this implies that $\mathbf{M}(\nu,t)>-\infty$ for all such $(\nu,t)$. For $(\nu,t)$ in a small neighbourhood of $(\mu,s)$ we therefore have $\int \|A\|^td\nu_n(A)>0$ for all $n \geq 1$, so by Theorem \ref{th:mpressure}
\[\mathbf{M}(\nu,t)\geq \frac{1}{n}\log\left(\frac{\int \|A\|^t d\nu_{nd}(A)}{K_{d,s}\left(\int \|A\|^t d\nu_{n}(A)\right)^{d-1}}\right)\]
for all $n \geq 1$ when $(\nu,t)$ is sufficiently close to $(\mu,s)$. Since the right-hand side converges to $\mathbf{M}(\nu,t)$ as $n \to \infty$ we deduce that
\[\mathbf{M}(\nu,t)=\sup_{n \geq 1}\frac{1}{n}\log \left(\frac{\int \|A\|^t d\nu_{nd}(A)}{K_{d,s}\left(\int \|A\|^t d\nu_{n}(A)\right)^{d-1}}\right)\]
for $(\nu,t)$ sufficiently close to $(\mu,s)$. This shows that in a neighbourhood of $(\mu,s)$ the function $\mathbf{M}$ is equal to a supremum of continuous functions, and hence is lower semi-continuous on that neighbourhood. In particular it is lower semi-continuous at $(\mu,s)$, and this completes the proof.


\section{Proof of Theorem \ref{th:fcontinuous}}\label{se:cuntcunt}

We begin the proof by deducing from Theorem \ref{th:mcontinuous} that $\mathbf{P}$ is continuous on sets of the form $\mathcal{M}_d\times\{s\}$ for $s\in\mathbb{Q}$:
\begin{lemma}\label{le:rational}
For every $d \geq 1$ and every rational number $s>0$, the function $\mu \mapsto \mathbf{P}(\mu,s)$ is a continuous function from $\mathcal{M}_d$ to $[-\infty,+\infty)$.
\end{lemma}
\begin{proof}
If $s \geq d$ then the result is trivial, since by the multiplicativity of the determinant
\[\mathbf{P}(\mu,s)=\lim_{n \to \infty}\frac{1}{n}\log\int |\det A|^{s/d}d\mu_n(A)=\log\int|\det A|^{s/d}d\mu(A)\]
in that case. For $s \leq 1$ the result follows from Theorem \ref{th:mcontinuous} since $\mathbf{P}(\mu,s)=\mathbf{M}(\mu,s)$ when $0<s \leq 1$. Let us therefore write $s=k+\frac{p}{q}$ where $k$ is an integer in the range $1 \leq k \leq d-1$, and $0 \leq \frac{p}{q}<1$.

Before proceeding with the proof we recall some facts from multilinear algebra. Given a finite-dimensional real inner product space $(V,\langle \cdot \rangle)$ and integer $n \geq 1$ we let $\otimes^n V$ denote the tensor product $V \otimes \cdots \otimes V$ of $n$ copies of $V$, which is an inner product space when equipped with the inner product defined by
\[\langle u_1\otimes \cdots \otimes u_n,v_1 \otimes \cdots \otimes v_n\rangle :=\prod_{i=1}^n \langle u_i,v_i\rangle \]
for rank-one elements of $\otimes^n V$ and by bilinear extension for general elements of $\otimes^n V$. If $e_1,\ldots,e_d$ is an orthonormal  basis for $V$, then the $d^n$ vectors $e_{i_1}\otimes \cdots \otimes e_{i_n}$ form an orthonormal basis for $\otimes^n V$. Given a linear map $A$ acting on $V$ we may define an induced linear map $A^{\otimes n}$ acting on $\otimes^nV$ via the relation $A^{\otimes n}(u_1\otimes \cdots \otimes u_n):=Au_1 \otimes \cdots \otimes Au_n$, and it follows from the definition of the inner product on $\otimes^n V$ that $(A^{\otimes n})^*=(A^*)^{\otimes n}$. These observations together imply that the $d^n$ singular values of $A^{\otimes n}$ are precisely the products $\sigma_{i_1}(A)\cdots \sigma_{i_n}(A)$ for $1 \leq i_1,\ldots,i_n\leq d$; in particular, in the induced norm on $\otimes^n V$ we have $\|A^{\otimes n}\|=\|A\|^n$ for all linear maps $A \colon V \to V$. In a similar fashion, given two linear maps $A$ and $B$ acting on not-necessarily-identical finite-dimensional inner product spaces $V$ and $W$ we may define their tensor product $A \otimes B$ acting on $V \otimes W$ in such a manner that $\|A\otimes B\|=\|A\|\cdot\|B\|$.

A linear map $A$ acting on a finite-dimensional real inner product space $V$ similarly induces a map $A^{\wedge k}$ on the $k^{\mathrm{th}}$ exterior power $\wedge^kV$ by $A^{\wedge k}(v_1 \wedge \cdots \wedge v_n):=(Av_1 \wedge \cdots \wedge Av_n)$ for rank-one elements and by linear extension for general elements. An inner product on $\wedge^kV$ is induced by the formula
\[\langle u_1 \wedge \cdots \wedge u_n,v_1\wedge \cdots \wedge v_n\rangle :=\det \left(\left[\langle u_i,v_j\rangle\right]_{i,j=1}^d\right).\]
If $e_1,\ldots,e_d$ is an orthonormal basis for $V$ then the set of all elements of $\wedge^kV$ of the form $e_{i_1} \wedge \cdots \wedge e_{i_k}$ with $1 \leq i_1<i_2<\cdots <i_k\leq d$ is an orthonormal basis for $\wedge^kV$. Thus the dimension of $\wedge^kV$ is $d \choose k$, and we will identify $\wedge^k\mathbb{R}^d$ with $\mathbb{R}^{d \choose k}$. By similar considerations to the preceding ones for the tensor power, it follows easily that for every $A \in M_d(\mathbb{R})$
\[\left\|A^{\wedge k}\right\|=\sigma_1(A)\sigma_2(A)\cdots \sigma_k(A)\]
where $\sigma_k(A)$ is understood to be zero if $k>d$.

We may now give the proof of the lemma. Given the integers $d,k$ and a measure $\mu \in \mathcal{M}_d$, define $\hat{d}:={d \choose k}^{q-p}{d \choose k+1}^p$, and identify the $\hat{d}$-dimensional inner product space $\left(\otimes^{q-p}\left(\wedge^k \mathbb{R}^d\right)\right) \otimes \left(\otimes^p \left(\wedge^{k+1}\mathbb{R}^d\right)\right)$ with the standard $\hat{d}$-dimensional inner product space $\mathbb{R}^{\hat{d}}$. Define a new measure $\hat\mu \in \mathcal{M}_{\hat{d}}$ by
\[\hat\mu(X):=\mu\left(\left\{A \in M_{d}(\mathbb{R})\colon (A^{\wedge k})^{\otimes (q-p)}\otimes (A^{\wedge (k+1)})^{\otimes p} \in X\right\}\right) \]
for all Borel measurable subsets $X$ of the closed unit ball of $M_{\hat d}(\mathbb{R})$.
Since the map $A \mapsto  (A^{\wedge k})^{\otimes (q-p)}\otimes (A^{\wedge (k+1)})^{\otimes p} $ is continuous it follows easily that the map $\mu \mapsto \hat\mu$ is continuous.  For each $n \geq 1$ we have
\begin{align*}\int \|A\|^{\frac{1}{q}} d\hat\mu_n(A)&=\int \left\|(A^{\wedge k})^{\otimes (q-p)}\otimes (A^{\wedge (k+1)})^{\otimes p}\right\|^{\frac{1}{q}}d\mu_n(A)\\
&=\int\left\|A^{\wedge k}\right\|^{\frac{q-p}{q}} \left\|A^{\wedge (k+1)}\right\|^{\frac{p}{q}}d\mu_n(A)\\
&=\int \sigma_1(A)\cdots \sigma_k(A)\sigma_{k+1}(A)^{\frac{p}{q}}d\mu_n(A) = \int \varphi^{k+\frac{p}{q}}(A)d\mu_n(A)\end{align*}
and therefore
\[\mathbf{M}\left(\hat\mu,\frac{1}{q}\right)=\mathbf{P}\left(\mu,s\right)\]
for all $\mu \in\mathcal{M}_d$. Since the maps $\mu \mapsto \hat\mu$ and $\hat\mu \mapsto \mathbf{M}(\hat\mu,1/q)$ are continuous (the latter in view of Theorem \ref{th:mcontinuous}) it follows that $\mu \mapsto \mathbf{P}(\mu,s)$ is continuous as claimed.
\end{proof}
We next note:
\begin{lemma}\label{le:decr}
For every $\mu \in \mathcal{M}_d$ the function $s \mapsto \mathbf{P}(\mu,s)$ is decreasing, and is continuous except possibly at $s=1,\ldots,d-1$. 
\end{lemma}
\begin{proof}
Fix $\mu \in \mathcal{M}_d$.  Clearly $\varphi^s(A) \geq \varphi^t(A)$ whenever $0<s \leq t$ and $\|A\|\leq 1$. It follows that in this case
\[\mathbf{P}(\mu,s) =\lim_{n \to \infty}\frac{1}{n}\log\int \varphi^s(A)d\mu_n(A)\geq \lim_{n \to \infty}\frac{1}{n}\log\int \varphi^t(A)d\mu_n(A)=\mathbf{P}(\mu,t)\]
and the function is decreasing as claimed.

For $s\geq d$ continuity is obvious. 
Let us show that $s \mapsto \mathbf{P}(\mu,s)$ is continuous on each of the intervals $(k,k+1]$, where $0 \leq k <d$. Indeed, fix such a $k$ and let $0<t_1<t_2\leq 1$ and $t:= \lambda t_1 + (1-\lambda)t_2 \in (t_1,t_2)$. For every $n \geq 1$ we have
\begin{align*}\int \varphi^{k+t}(A)d\mu_n(A) &= \int \sigma_1(A)\cdots \sigma_k(A)\sigma_{k+1}(A)^{t}d\mu_n(A)\\
& =\int \varphi^{k+t_1}(A)^{\lambda}\varphi^{k+t_2}(A)^{1-\lambda}d\mu_n(A)\\
&\leq \left(\int \varphi^{k+t_1}(A)d\mu_n(A)\right)^{\lambda} \left(\int \varphi^{k+t_2}(A)d\mu_n(A)\right)^{1-\lambda}\end{align*}
by H\"older's inequality. By taking logarithms and letting $n \to \infty$ we deduce that
\[\mathbf{P}\left(\mu,k+t\right) \leq \lambda \mathbf{P}(\mu,k+t_1)+(1-\lambda)\mathbf{P}(\mu,k+t_2),\]
and thus the convexity property
\[\mathbf{P}\left(\mu,\lambda s_1 + (1-\lambda)s_2\right) \leq \lambda \mathbf{P}(\mu,s_1)+(1-\lambda)\mathbf{P}(\mu,s_2)\]
is satisfied for all $s_1,s_2 \in (k,k+1]$ and $\lambda \in [0,1]$. By standard results from convex analysis (see e.g. \cite[Theorem 10.1]{Ro70}) it follows that $s \mapsto \mathbf{P}(\mu,s)$ is continuous on $(k,k+1)$. Since $s \mapsto \varphi^s(A)$ is upper semi-continuous for each fixed $A$, the map $s \mapsto \mathbf{P}(\mu,s)=\inf_{n \geq 1}\frac{1}{n}\log \int \varphi^s(A)d\mu_n(A)$ is the infimum of a sequence of upper semi-continuous functions and hence is upper semi-continuous; since moreover it is decreasing, this implies that it is continuous from the left, and hence is continuous on $(k,k+1]$. This completes the proof.
\end{proof}
The above observations already suffice to prove part (i) of Theorem \ref{th:fcontinuous}:
\begin{proof}[Proof of Theorem \ref{th:fcontinuous}(i)]
Since $(s,A)\mapsto \varphi^s(A)$ is upper semi-continuous it follows easily via \eqref{eq:naffer} that the function $(\mu,s) \mapsto \mathbf{P}(\mu,s) $ is an infimum of upper semi-continuous functions from $\mathcal{M}_d\times(0,+\infty)$ to $[-\infty,+\infty)$, and hence is upper semi-continuous. It is also continuous on $\mathcal{M}_d\times[d,+\infty)$ as a consequence of the formula $\mathbf{P}(\mu,s)=\log\int |\det A|^{s/d}d\mu(A)$ which is valid in that region. It therefore suffices to prove the lower semi-continuity of $\mathbf{P}$ on $\mathcal{M}_d\times (k,k+1]$ for every integer $k$ such that $0\leq k<d$.

Let $(\mu,s) \in \mathcal{M}_d \times (k,k+1]$ where $0 \leq k <d$. If $\mathbf{P}(\mu,s)=-\infty$ then lower semicontinuity at $(\mu,s)$ holds trivially, so we assume this not to be the case. Suppose firstly that $s<k+1$. Given $\varepsilon>0$, using Lemma \ref{le:decr} we may choose $s_0 \in (k,k+1]\cap\mathbb{Q}$ such that $s<s_0$ and $\mathbf{P}(\mu,s_0)>\mathbf{P}(\mu,s)-\varepsilon/2$. Using Lemma \ref{le:rational} we may choose an open neighbourhood $U$ of $\mu$ such that $\mathbf{P}(\nu,s_0)>\mathbf{P}(\mu,s_0)-\varepsilon/2$ for all $\nu \in U$. It follows via Lemma \ref{le:decr} that for all $(\nu,t) \in U \times (k,s_0)$ we have
\[\mathbf{P}(\nu,t)\geq \mathbf{P}(\nu,s_0)>\mathbf{P}(\mu,s_0)-\frac{\varepsilon}{2}>\mathbf{P}(\mu,s)-\varepsilon\] and therefore $\mathbf{P}$ is lower semi-continuous at $(\mu,s)$. If $s=k+1$, we may similarly apply Lemma \ref{le:rational} to choose an open neighbourhood $U$ of $\mu$ such that $\mathbf{P}(\nu,k+1)>\mathbf{P}(\mu,k+1)-\varepsilon$ for all $\nu \in U$, which yields $\mathbf{P}(\nu,t) > \mathbf{P}(\mu,s)-\varepsilon$ for all $(\nu,t) \in U \times (k,k+1]$ in a similar manner.\end{proof}

To prove the remainder of the theorem we require an additional result. For the proof of the theorem we require only the case $p=0$, but we include the more general case for the benefit of the discussions in \S\ref{ss:one} and \S\ref{ss:two} above.
\begin{proposition}\label{pr:flight}
Let $k,d$ be integers such that $0<k<d$ and let $s=k+\frac{p}{q} \in (0,d)$. Let $\mu$ be a Borel measure on $M_d(\mathbb{R})$ such that $\int \varphi^k(A)d\mu(A)<\infty$.Then
\begin{equation}\label{eq:cutprice}\int \varphi^s(A)d\mu_{d'}(A) \leq Ke^{\mathbf{P}(\mu,s)} \left(\int \varphi^s(A)d\mu(A)\right)^{d'-1}\end{equation}
where $d':={d\choose k}^{q-p}{d \choose k+1}^p$ and $K:=(d')^{2+(d'+1)/q}(d'+1)^{(q-1)/q}$.
\end{proposition}
\begin{proof}
Similarly to Lemma \ref{le:rational} we identify $\left(\otimes^{q-p}\left(\wedge^k \mathbb{R}^d\right)\right) \otimes \left(\otimes^p \left(\wedge^{k+1}\mathbb{R}^d\right)\right)$ with $\mathbb{R}^{d'}$ and define a measure $\hat\mu$ on  $M_{d'}(\mathbb{R})$ by
\[\hat\mu(X):=\mu\left(\left\{A \in M_{d}(\mathbb{R})\colon (A^{\wedge k})^{\otimes (q-p)}\otimes (A^{\wedge (k+1)})^{\otimes p} \in X\right\}\right) \]
for all Borel measurable sets $X \subseteq M_{d'}(\mathbb{R})$. Since
\[\left\|(A^{\wedge k})^{\otimes (q-p)}\otimes (A^{\wedge (k+1)})^{\otimes p}\right\|^{\frac{1}{q}}\equiv \sigma_1(A)\sigma_2(A)\cdots \sigma_k(A)\sigma_{k+1}(A)^{\frac{p}{q}} \equiv \varphi^s(A)\]
we have $\mathbf{P}(\mu,s)=\mathbf{M}(\hat\mu,1/q)$ and hence by Theorem \ref{th:mpressure}
\[\int \|A\|d\hat\mu_{d'}(A)\leq  K e^{\mathbf{M}(\hat\mu,1/q)}\left(\int \|A\|d\hat\mu(A)\right)^{d'-1}\]
which is precisely \eqref{eq:cutprice}.
\end{proof}
\begin{proof}[Proof of Theorem \ref{th:fcontinuous}(ii)]
Let $0<k<d$ and $\mu \in \mathcal{M}_d$. If $\mathbf{P}(\mu^0,k)=-\infty$ then in view of the upper semi-continuity of $\mathbf{P}$ nothing remains to be proved, so we shall assume that this is not the case. 
Let $d'$, $K$ be as given by Proposition \ref{pr:flight}; since $\mathbf{P}(\mu^0,k)>-\infty$ we have $\int \varphi^k(A)d\mu_{d'}^0(A)>0$. By Theorem \ref{th:fcontinuous}(i) the function $(\nu,s) \mapsto \mathbf{P}(\nu,s)$ is continuous on $\mathcal{M}_d \times (k-1,k]$, so to prove (ii) it is sufficient to prove
\[\liminf_{\substack{\nu \to \mu\\s \to k^+}} \mathbf{P}(\nu,s) \geq \mathbf{P}(\mu^0,k).\]
For each $\varepsilon>0$ let $K_\varepsilon:=\{A \in M_d(\mathbb{R}) \colon \sigma_d(A) \geq \varepsilon\}$, and let $U_\varepsilon:=\{A \in M_d(\mathbb{R}) \colon \sigma_d(A)>\varepsilon\}$. For each $\nu \in \mathcal{M}_d$ and $\varepsilon>0$ let us define two measures $\nu^\varepsilon$, $\overline{\nu}^\varepsilon$ on the closed unit ball of $M_d(\mathbb{R})$ by $\nu^\varepsilon(X):=\nu(X \cap U_\varepsilon)$ and $\overline{\nu}^\varepsilon(X):=\nu(X \cap K_\varepsilon)$ for every Borel measurable subset $X$ of the closed unit ball of $M_d(\mathbb{R})$. For each $\nu \in \mathcal{M}_d$, $\varepsilon>0$ and $s \in [k,k+1)$ we have
\begin{align*}\mathbf{P}(\nu,s) \geq \mathbf{P}(\overline{\nu}^\varepsilon,s)&=\lim_{n \to \infty} \frac{1}{n}\log\left(\int\sigma_1(A)\cdots \sigma_k(A)\sigma_{k+1}(A)^{s-k}d\overline{\nu}^\varepsilon_n(A)\right)\\
&\geq \lim_{n \to \infty} \frac{1}{n}\log\left(\int \sigma_1(A)\cdots \sigma_k(A)\varepsilon^{n(s-k)}d\overline{\nu}^\varepsilon_n(A)\right) \\
&= (s-k)\log \varepsilon+ \mathbf{P}\left(\overline{\nu}^\varepsilon,k\right)\end{align*}
and hence for each $\varepsilon>0$
\[\liminf_{\substack{\nu \to \mu\\ s \to k^+}} \mathbf{P}(\nu,s) \geq \liminf_{\nu \to \mu} \mathbf{P}(\overline{\nu}^\varepsilon ,k).\]
Since $\int \varphi^k(A)d\mu^0_{d'}(A)>0$ we may by the monotone convergence theorem choose $\varepsilon_0>0$ such that for all $\varepsilon \in (0,\varepsilon_0)$ we have $\int \varphi^k(A)d\mu^\varepsilon_{d'}(A)>0$. Since each $U_\varepsilon$ is open, it follows from standard properties of the weak-* topology on $\mathcal{M}_d$ that for all $\varepsilon \in (0,\varepsilon_0)$ we have $\int \varphi^k(A)d\nu^\varepsilon_{d'}(A)>0$ whenever $\nu \in \mathcal{M}_d$ is sufficiently close to $\mu$. In particular, if $0<\varepsilon<\varepsilon_0$ then by Proposition \ref{pr:flight} we have $\mathbf{P}(\nu^\varepsilon,k)>-\infty$ when $\nu$ is sufficiently close to $\mu$.

Let $0<2\varepsilon<\varepsilon_0$ and suppose that $\nu$ is so close to $\mu$ that $\mathbf{P}(\nu^{2\varepsilon},k)>-\infty$. We have $\mathbf{P}(\overline{\nu}^{\varepsilon},k) \geq \mathbf{P}(\nu^{2\varepsilon},k)>-\infty$, so $\int \varphi^k(A)d\overline{\nu}_n^{\varepsilon}(A)>0$ for every $n \geq 1$ and therefore using Proposition \ref{pr:flight}
\begin{align*}\mathbf{P}(\overline{\nu}^{\varepsilon},k) &=\sup_{n \geq 1}\frac{1}{n}\log\left(\frac{\int \varphi^k(A)d\overline{\nu}_{nd'}^{\varepsilon}(A)}{K\left(\int \varphi^k(A)d\overline{\nu}_{n}^{\varepsilon}(A)\right)^{d'-1}}\right)\\
&\geq \sup_{n \geq 1}\frac{1}{n}\log\left(\frac{\int \varphi^k(A)d\nu_{nd'}^{2\varepsilon}(A)}{K\left(\int \varphi^k(A)d\overline{\nu}_{n}^{\varepsilon}(A)\right)^{d'-1}}\right).\end{align*}
Now, since $U_{2\varepsilon}$ is open and $K_\varepsilon$ is closed,
\[\liminf_{\nu \to \mu} \int\varphi^k(A)d\nu^{2\varepsilon}_{nd'}(A) \geq \int \varphi^k(A) d\mu^{2\varepsilon}_{nd'}(A)\]
and
\[\limsup_{\nu \to \mu} \int\varphi^k(A)d\overline{\nu}^{\varepsilon}_{nd'}(A) \leq \int \varphi^k(A) d\overline{\mu}^{\varepsilon}_{nd'}(A)\]
for every $n \geq 1$. Hence,
\begin{align*}\liminf_{\nu \to \mu} \mathbf{P}\left(\overline{\nu}^{\varepsilon},k\right) &\geq \liminf_{\nu \to \mu} \sup_{n \geq 1}\frac{1}{n}\log\left(\frac{\int \varphi^k(A)d\nu_{nd'}^{2\varepsilon}(A)}{K\left(\int \varphi^k(A)d\overline{\nu}_{n}^{\varepsilon}(A)\right)^{d'-1}}\right)\\
& \geq \sup_{n \geq 1}\liminf_{\nu \to \mu}\frac{1}{n}\log\left(\frac{\int \varphi^k(A)d\nu_{nd'}^{2\varepsilon}(A)}{K\left(\int \varphi^k(A)d\overline{\nu}_{n}^{\varepsilon}(A)\right)^{d'-1}}\right)\\
& \geq \sup_{n \geq 1}\frac{1}{n}\log\left(\frac{\int \varphi^k(A)d\mu_{nd'}^{2\varepsilon}(A)}{K\left(\int \varphi^k(A)d\overline{\mu}_{n}^{\varepsilon}(A)\right)^{d'-1}}\right)\\
& \geq \sup_{n \geq 1}\frac{1}{n}\log\left(\frac{\int \varphi^k(A)d\mu_{nd'}^{2\varepsilon}(A)}{K\left(\int \varphi^k(A)d\mu_{n}^{0}(A)\right)^{d'-1}}\right).
\end{align*}
Since by the monotone convergence theorem
\[\lim_{\varepsilon \to 0} \int \varphi^k(A)d\mu_{nd'}^{2\varepsilon}(A)= \int \varphi^k(A)d\mu_{nd'}^0(A)\]
for every $n \geq 1$, we conclude that
\begin{align*}\lim_{\substack{\nu \to \mu\\ s \to k^+}} \mathbf{P}(\nu,s) &\geq \lim_{\varepsilon \to 0}\liminf_{\nu \to \mu} \mathbf{P}(\overline{\nu}^\varepsilon,k)\\
& \geq \sup_{n \geq 1}\frac{1}{n}\log\left(\frac{\int \varphi^k(A)d\mu_{nd'}^{0}(A)}{K\left(\int \varphi^k(A)d\mu_{n}^{0}(A)\right)^{d'-1}}\right)=\mathbf{P}(\mu^0,k)\end{align*}
and this completes the proof.

\end{proof}


\section{Proof of Theorem \ref{th:cks}}

Whilst Theorem \ref{th:cks} can be deduced from the full strength of Theorem \ref{th:fcontinuous}(i) in a similar manner to \cite{FeSh14}, we will give a proof using Lemma \ref{le:rational} alone. 
For each $A_1,\ldots,A_N \in \mathbf{B}_d$ and $s>0$ let us define
\begin{align*}\mathbf{P}((A_1,\ldots,A_N),s)&:=\lim_{n \to \infty}\frac{1}{n}\log \sum_{i_1,\ldots,i_n=1}^N \varphi^s(A_{i_1}\cdots A_{i_n})\\
&=\mathbf{P}\left(\sum_{i=1}^N\delta_{A_i},s\right)=s\log N+\mathbf{P}\left(\frac{1}{N}\sum_{i=1}^N\delta_{A_i},s\right).\end{align*}
Since the map $\mathbf{B}_d^N \to \mathcal{M}_d$ defined by $(A_1,\ldots,A_N) \mapsto \frac{1}{N}\sum_{i=1}^N\delta_{A_i}$ is continuous, it follows from Lemma \ref{le:rational} that the above expression depends continuously on $(A_1,\ldots,A_N)$ when $s>0$ is rational.

For fixed $A_1,\ldots,A_N$ let $\theta:=\max\|A_i\|\in [0,1)$. Clearly we have
\[e^{\mathbf{P}((A_1,\ldots,A_N),s+t)}\leq\theta^te^{\mathbf{P}((A_1,\ldots,A_N),s)}\]
 for every $s,t>0$. It follows in particular that the function $s \mapsto \mathbf{P}((A_1,\ldots,A_N),s)$ is decreasing and is not constant on any interval on which it takes a finite value.  Using these observations we may characterise $\mathfrak{s}$ by the expression
\begin{align*}\mathfrak{s}(A_1,\ldots,A_N)&=\inf\left\{s>0 \colon \mathbf{P}((A_1,\ldots,A_N),s)<0\right\}\\
&=\sup\left\{s>0 \colon \mathbf{P}((A_1,\ldots,A_N),s)>0\right\}.\end{align*}
To prove the theorem it is sufficient to show that for all $s>0$, the sets $\mathfrak{s}^{-1}([0,s))$ and $\mathfrak{s}^{-1}((s,+\infty))$ are open: but by the above characterisation we may write
\[\mathfrak{s}^{-1}([0,s)) = \bigcup_{t \in (0,s) \cap \mathbb{Q}} \{(A_1,\ldots,A_N) \in \mathbf{B}_d^N \colon \mathbf{P}((A_1,\ldots,A_N),t)<0\},\]
\[\mathfrak{s}^{-1}((s,+\infty)) = \bigcup_{t \in (s,+\infty) \cap \mathbb{Q}} \{(A_1,\ldots,A_N) \in \mathbf{B}_d^N \colon \mathbf{P}((A_1,\ldots,A_N),t)>0\},\]
and since $(A_1,\ldots,A_N)\mapsto \mathbf{P}((A_1,\ldots,A_N),t)$ is continuous for rational $t$, these are open sets.

\section{Proof of Theorem \ref{th:fcontinuous2}}

By Theorem \ref{th:fcontinuous} we know that $\mathbf{P}$ is continuous at $(\mu,s)$ if $s\neq 1$ or if $s=1$ and $\mathbf{P}(\mu,1)=\mathbf{P}(\mu^0,1)$, so it is sufficient to show that if $\mathbf{P}(\mu,1)>\mathbf{P}(\mu^0,1)$ then $\mathbf{P}$ is discontinuous at $(\mu,1)$. Let us assume this inequality to be satisfied. We observe that $\varphi^t(A)=0$ when $t>1$ and $|\det A|=0$, so for $t>1$ we have $\int \varphi^t(A)d\mu_n(A)=\int \varphi^t(A)d\mu_n^0(A)$ for all $n \geq 1$, and in particular $\mathbf{P}(\mu,t)=\mathbf{P}(\mu^0,t)$ for every $t>1$. Observe also that $\mathbf{P}(\nu,1)=\mathbf{M}(\nu,1)$ for every $\nu \in \mathcal{M}_2$, and that $\mathbf{P}(\nu,t) \leq \mathbf{M}(\nu,t)$ for all $(\nu,t) \in \mathcal{M}_2 \times (0,+\infty)$ as a consequence of the elementary inequality $\varphi^t(A) \leq \|A\|^t$. Assembling these parts we compute that 
\[\limsup_{t \to 1^+} \mathbf{P}(\mu,t) =\limsup_{t \to 1^+} \mathbf{P}(\mu^0,t)\leq \lim_{t \to 1^+} \mathbf{M}(\mu^0,t) =\mathbf{M}(\mu^0,1)=\mathbf{P}(\mu^0,1)<\mathbf{P}(\mu,1) \]
where we have used Theorem \ref{th:mcontinuous}, and we have shown that $\mathbf{P}$ is discontinuous at $(\mu,1)$.


\section{Proof of Theorem \ref{th:fpressure}}
When $0 <s\leq 1$ we have $\mathbf{P}(\mu,s)=\mathbf{M}(\mu,s)$ and the result is immediate from Theorem \ref{th:mpressure}; when $s \geq 2$ we simply have
\[e^{\mathbf{P}(\mu,s)}=\lim_{n \to \infty} \left(\int |\det A|^{\frac{s}{2}}d\mu_n(A)\right)^{\frac{1}{n}}=\int |\det A|^{\frac{s}{2}}d\mu(A)\]
by the multiplicativity of the determinant, so the result is trivial. If $1<s<2$ then we define a new measure $\hat\mu$ on $M_2(\mathbb{R})$ by $\int f(A)d\hat\mu(A):=\int f(A)|\det A|^{s-1}d\mu(A)$ for all compactly supported continuous $f \colon M_2(\mathbb{R}) \to \mathbb{R}$. Since for each $n \geq 1$
\begin{align*}\int \varphi^s(A)d\mu_n(A)&=\int \sigma_1(A)\sigma_2(A)^{s-1}d\mu_n(A) \\
&= \int \|A\|^{2-s}|\det A|^{s-1}d\mu_n(A) = \int \|A\|^{2-s}d\hat\mu_n(A)\end{align*}
we have $\mathbf{P}(\mu,s)=\mathbf{M}(\hat\mu,2-s)$, so using Theorem \ref{th:mpressure}
\begin{align*}\int \varphi^s(A)d\mu_2(A)&= \int \|A\|^{2-s}d\hat\mu_2(A)\\
& \leq K_{d,2-s}e^{\mathbf{M}(\mu,2-s)} \int \|A\|^{2-s}d\hat\mu(A)\\
&=2^{7-2s} e^{\mathbf{P}(\mu,s)} \int \varphi^s(A)d\mu(A)\end{align*}
as required, where in the first line we have exploited the multiplicativity of the determinant. The proof is complete.


\section{Proof of Theorem \ref{th:limit}}\label{se:last}

For each $s>0$ and $n \geq 1$ we have
\[\left(\int \|A\|^sd\mu_n(A)\right)^{\frac{1}{ns}} \leq \mu(\mathsf{A})^{\frac{1}{s}}\mathrm{ess}\text{ }{\sup}_{\mu_n} \|A\|^{\frac{1}{n}} = \mu(\mathsf{A})^{\frac{1}{s}}\sup_{A_1,\ldots,A_n \in \mathsf{A}}\|A\|^{\frac{1}{n}}\]
since $\mu$ has support equal to $\mathsf{A}$. Taking the limit as $n \to \infty$ yields
\[e^{\mathbf{M}(\mu,s)/s} \leq \mu(\mathsf{A})^{\frac{1}{s}} \varrho_\infty(\mathsf{A})\]
so that 
\[\limsup_{s \to \infty}e^{\mathbf{M}(\mu,s)/s} \leq \varrho_\infty(\mathsf{A}).\]
If $\varrho_\infty(\mathsf{A})=0$ then we are done. Otherwise, it must be the case that for each $n \geq 1$ there exists a nonzero product $A_1\cdots A_n$ with $A_1,\ldots,A_n \in \mathsf{A}$, and since $\mathsf{A}$ is the support of $\mu$ it follows that $\int \|A\|^s d\mu_n(A)>0$ for every $s>0$ and $n \geq 1$. Using Theorem \ref{th:mpressure} we may therefore write
\[e^{\mathbf{M}(\mu,s)/s} \geq \left(\frac{\int \|A\|^s d\mu_{nd}(A)}{K_{d,s}\left(\int \|A\|^s d\mu_{n}(A)\right)^{d-1}}\right)^{\frac{1}{ns}}\]
for every $s>0$ and $n \geq 1$. Taking the limit as $s \to \infty$ yields
\[\liminf_{s \to \infty} e^{\mathbf{M}(\mu,s)/s} \geq \left(\frac{\sup_{B_1,\ldots,B_{nd} \in \mathsf{A}} \|B_1\cdots B_{nd}\|}{d^{d+1}\left(\sup_{A_1,\ldots,A_n \in \mathsf{A}} \|A_1\cdots A_n\|\right)^{d-1}}\right)^{\frac{1}{n}}\]
and the limit $n \to \infty$ yields $\liminf_{s \to \infty}e^{\mathbf{M}(\mu,s)/s} \geq \varrho_\infty(\mathsf{A})$ as required to complete the proof.

\section{acknowledgments}
This research was supported by EPSRC grant  EP/L026953/1.
The author thanks P. Shmerkin, V. Yu. Protasov and R. Jungers for helpful remarks.

\bibliographystyle{siam}
\bibliography{cont}

\end{document}